\documentclass[12pt]{article}
\usepackage[english]{babel}
\usepackage{latexsym}
\usepackage{amsfonts, amsmath, amssymb, amsthm}
\usepackage[dvips]{graphicx}
\usepackage{indentfirst}
\usepackage{epsfig}
\usepackage{bm}

\DeclareMathOperator*{\rot}{\rot}

\renewcommand{\rot}{\mathop{\mathrm{rot}}\nolimits}

\renewcommand{\Im}{\mathop{\mathrm{Im}}\nolimits}
\renewcommand{\Re}{\mathop{\mathrm{Re}}\nolimits}

\theoremstyle{plain}\newtheorem{theorem}{Theorem}
\theoremstyle{plain}
\theoremstyle{plain}
\theoremstyle{plain}
\theoremstyle{definition}\newtheorem{definition}{Definition}
\theoremstyle{remark}\newtheorem{remark}{Remark}

\begin{document}
{\Large
\begin{center}
\textbf{Inverse Problem of Diffraction by an Inhomogeneous Solid with a Piecewise H\"{o}lder Refractive Index}\\
M.Yu. Medvedik, Yu.G. Smirnov and A.A. Tsupak\\
\end{center}
}
\begin{center}
(Penza State University, Penza, Russia)\\
e-mail: smirnovyug@mail.ru; altsupak@yandex.ru\\
\end{center}
%
%
%

\section*{Introduction}
Nowadays, when solving inverse problems of electrodynamics, one meets a problem of obtaining accurate solutions using a relatively small number of measurements. The present work is devoted to the solution of this very problem.

The urgency of the problem solved is due, on the one hand, to important applications, e.g., microwave tomography for early of breast cancer and  reconstruction of the characteristics of samples of composite anisotropic materials.

On the other hand, methods and algorithms for solving inverse problems are important from the theoretical point of view in mathematical physics.

New numerical adaptive methods, parallel computational algorithms and supercomputer computations are among the most promising approaches to solving the inverse problems.

In particular, there are several approaches and many papers devoted to solving inverse problems of electrodynamics. Among theoretical works, in which the questions of the existence and uniqueness of solutions to the inverse problems are considered, we only mention the work \cite{a_Brown}, since it contains an extensive and up-to-date bibliography on the subject (63 references).

The most important methods for solving inverse problems are those that can be numerically implemented. These methods are mainly based on solving the hyperbolic systems of differential equations in the time domain using the finite differences or finite elements methods, with subsequent minimization of the corresponding functionals and Tikhonov regularization. Such methods and approaches are described in the monographs \cite{b_Ammari} - \cite{b_Romanov}.

The volume singular integral equations method is an alternative approach to solving the inverse problems of electrodynamics. It is not widely used in electrodynamics (apparently, due to the greater complexity of numerical implementation in comparison with the finite element or finite difference methods. Nevertheless, it is successfully applied by both the authors of the paper and other researchers (e.g., \cite{b_ColtonKress_Inv}--\cite {a_Smirnov}) to solving (direct) problems of diffraction of electromagnetic waves by dielectric solids.

The method of volume singular integral equations was applied, e.g., in \cite{b_ColtonKress_Inv} and \cite{a_ShestopalovSmirnov}, where the problems of  objects’ reconstruction were solved using asymptotic data (far zone fields).

In the present paper, we use the method of integral equations for reconstruction of an unknown refractive index $k(x)$ (this function describes an inhomogeneous volume obstacle $P$ of a monochromatic wave). To solve the problem, we use a finite number of values of the given scalar field in the near zone, i.e., at several points in some region $D,$ lying outside the scatterer $P.$ Problems exactly like this one arise, e.g., in microwave tomography.

The article consists of two parts, the first of which is devoted to theoretical investigation of the problem and the description of the proposed method for its solution. The second part contains the description of computational experiments. Solving the inverse problem we pay considerable attention to the investigation of the direct diffraction problem. This is necessary for proving the equivalency between the differential and the integral formulations (the latter is used to solve the inverse problem) of the inverse problem of reconstructing the refractive index.

The first section of the article is devoted to the investigation of a direct problem diffraction of an external field by a volume body $P.$ The solid $P$ is characterized by a given refractive index.

First, we consider a boundary value problem for the Helmholtz equation in the classical formulation. Then, the original problem is reduced to the integral Lippman-Schwinger equation with respect to the unknown total field $u$ in the inhomogeneity region $P.$ In addition, we define $u$ outside $P$ using the integral representation. We assume that the incident wave is defeined by a point source located outside $P,$ whereas the refractive index is assumed to be a piecewise-Holder function.

The Lippmann-Schwinger equation is convenient to be studied in the $L_2 (P)$ space for two reasons. First, in this case the operator of the equation is a  Fredholm operator with index zero. Second, we are eager to use piecewise constant basis functions for numerical implementation of the method.

Investigating the integral equation in such a wide space, we show that, given a sufficiently smooth right-hand side, one obtains a smooth solution to the equation which, furthermore, represents a classical solution of the diffraction problem.

Thus, we show that the boundary value problem for the Helmholtz equation is equivalent to a system of integral equations, which is used to reconstruct the refractive index in the second section of the article.

We find the solution to the inverse problem using the two-step method (TSM). First, we determine the "current" $ J = (k^2 - k_0^2)u$ in the inhomogeneity region $P$ using the values of the total field given in the domain $D.$ For this end, we solve the integral equation of the first kind. It is shown in the paper that the solution to the integral equation is, in general, not unique. However, we prove that the solution is unique in the class of piecewise-constant functions. Such a choice of the solution space is seemingly valid for solving the inverse problem. Indeed, we are eager to use the collocation method with piecewise constant basis functions with compact support. In addition, the complete field $u$ corresponding to the piecewise-constant $J$ satisfies all the smoothness conditions that were formulated in the direct problem. That is due  to the representation of the field $u$ via the volume potential with the density $J.$

In the second step of the proposed method, we express the desired function $k (x) $ via the current $ J (x) $ and the incident wave $u_0$ using the Lippman-Schwinger equation in the domain $P.$

In the proposed method, we immediately define (and fixate) the computation grid. Therefore, classical theorems on the convergence of the method are not considered in the article. We believe that such an approach is valid in practice since it is usually clear what accuracy (and, consequently, what grid) is needed. E.g., it is sufficient to consider grids with  5mm step in the method of microwave tomography for the breast cancer early diagnosis (thus, any grid refinement is no use).

The main advantage of the proposed method is the reduction of the original boundary value problem to the solution of the \ emph {linear} integral equation with the subsequent explicit calculation of the desired function. Thus, we are spared, first, the need to solve any nonlinear equations. Secondly, we don’t apply iterative methods that require good choice of an initial approximation.

One of the main results of the present paper is the proof of the uniqueness theorem for the solution to the integral equation of the first kind in the class of piecewise-constant functions. In addition, we show that even in the class of analytic functions the solution is not unique.

The main difficulty of the numerical method implementation is in solving the integral equation of the first kind with a smooth kernel. Note, that we apply collocation method. The corresponding matrices are ill-conditioned (albeit nondegenerate), which implies a certain instability with respect to the right-hand side of the equation. As a result, given distorted original data (i.e., the given values of the field in the domain $D$), one obtains "false" inhomogeneities in the region $P.$  In this paper, such inhomogeneities are called "artifacts" of the reconstructed function $k(x).$ The computational experiments showed that the "artifacts" emerged when the  domain $D$ was moved away from the body $P.$

The authors propose two ways to eliminate the "artifacts."

The first method consists in the screening out "extraneous" noise, i.e. in filtering the input data of the integral equation. When implementing this procedure, we assume that the "true" noise should be a small (but sufficiently smooth) disturbance of the field. Data that does not satisfy this condition is discarded.

The second method consists in changing the location of the field source (another choice of the point $ x_0 $) and the receivers (these are the collocation points in the domain $ D $). As a result, we determine the true inhomogeneities in the region $P$ and discard the false ones (the "artifacts"). Note that the “artifacts”  may change their location. In addition, the "artifacts" are located in the $ P $ region quite symmetrically which indirectly confirms their "artificial" nature. Application of the rotation method allows to restore the inhomogeneity of the body with sufficient accuracy.

To refine the solution, we apply the following adaptive method: in the region of the inhomogeneity found, we define new (and more dense) grid, then carry out additional measurements. As a result, a more accurate value of the refractive index is obtained.

The authors solved a series of problems, analyzed the solutions obtained and compared them with the exact solutions. A description of one of the numerical experiments is presented in this article.

\section{The Direct Scattering Problem}

Prior to the describing the statement of the inverse problem as well as the method for its solving, we shall investigate the direct diffraction problem.

\subsection{The Boundary Value Problem of Diffraction by a Solid with a Piecewise H\"{o}lder Refractive Index}

Consider an isotropic inhomogeneous rectangle parallelepiped
$$P = \{x=(x_1,x_2,x_3): a_1 < x_1 < b_1 ,\;a_2 < x_2 < b_2 ,\;a_3 < x_3 < b_3 \}$$
located in the homogeneous space $\mathbb R^3.$

Define the uniform mesh in the domain $\overline P$ as follows
$$
x_{1,i_1} = a_1 + \frac{b_1 - a_1 }{n}i_1,\;x_{2,i_2} = a_2 + \frac{b_2-a_2 }{n}i_2,\;x_{3,i_3} = a_3 + \frac{b_3-a_3}{n}i_3, \quad (0\le i_k\le n)
$$
and introduce sub-domains $\Pi_{i_1i_2i_3}:$
$$
\Pi_{i_1i_2i_3} = \{x: x_{k,i_k} < x_k < x_{k,k+1}\},\quad 0\le i_k\le n-1.
$$

We also define a set of piecewise constant functions $\chi_{i_1i_2i_3}$ (indicator functions):
\begin{equation}\label{sip10}
\chi_{i_1i_2i_3}(x) =
\begin{cases}
 1,& x\in  \Pi_{i_1i_2i_3},\\
 0,& x\notin \Pi_{i_1i_2i_3}.
\end{cases}
\end{equation}

Henceforward we assume that inhomogeneity of the domain $P$ is described by a piecewise continuous function $k(x)=n(x)k_0$ such that
\begin{equation}\label{sip20}
k(x) = \left\{ k_{i_1i_2i_3}(x),\quad x\in \Pi_{i_1i_2i_3}\right.,
\end{equation}
where all functions $k_{i_1i_2i_3}(x)$ are H\"{o}lder continuous
$$k_{i_1i_2i_3}\in C^{0,\alpha}( \Pi_{i_1i_2i_3}).$$
Note that at the points of the parallelepipedal faces $\partial \Pi_{i_1i_2i_3}$ the function $k(x)$ can be defined via one-sided limits from either side of the face.

By introducing the multi-indices $I=(i_1i_2i_3)$ one can now define the function $k(x)$ for any point  $x\in P$ by the following equality:
$$k(x)=\sum_Ik_I(x)\chi_I(x).$$

The lossless medium outside the solid is characterized by a given positive wavenumber $k_0>0.$

Define $E_P$ as the union of all edges of the parallelepipeds $\Pi_I$ and give the following notation:
\begin{equation*}
\Pi'_I=\Pi_I\setminus E_P,\quad P'=P\setminus E_P.
\end{equation*}

The incident wave (the source field) as well as the scattered and the total fields are considered to depend on time harmonically:
\begin{equation}\label{sip30}
U_0(x,t)=u_0e^{-i\omega t},\;U_s(x,t)=u_se^{-i\omega t},\; U(x,t)=U_0(x,t)+U_s(x,t),
\end{equation}
Thus it is sufficient to formulate the scattering problem for the scalar complex amplitude  $u(x)$ of the total field.

We consider the incident field of a point source setting
\begin{equation}\label{sip35}
u_0(x)=\frac{e^{ik_0|x-x_0|}}{4\pi |x-x_0|},\quad x_0\notin \overline P.
\end{equation}
The field represents a solution to the Helmholtz equation
$$
(\triangle + k_0^2) u_0(x) = -\delta(x-x_0)
$$
that satisfies Sommerfeld radiation conditions.

The direct scattering problem in the rigorous mathematical statement is to find a solution $u(x)$ to the following boundary value problem:
\begin{equation*}
({\cal P}_1)\quad
\begin{cases}
    (\triangle + k^2_I(x))u(x)= 0,\quad x\in \Pi_I;\quad\quad (\triangle + k_0^2(x))u(x)=-\delta(x-x_0),\quad x\in \mathbb R^3\setminus (\overline P\cup \{x_0\});\\
    \left.[u]\right|_{\partial \Pi_I} = 0, \left.\left[\frac{\partial u}{\partial \mathbf n} \right]\right|_{\partial \Pi'_I} = 0;\\
    u \in H_{loc}^1(\mathbb R^3\setminus\{x_0\});\\
    \frac{\partial u_s}{\partial r} = i k_0 u_s + o\left( {\frac{1}{r}} \right), \; (\Im k_0 = 0);\quad  u_s(r)= O\left(\frac{1}{r^2}\right),\;(\Im k_0 > 0).\\
  \end{cases}
\end{equation*}

\begin{definition}
Any solution to the problem $({\cal P}_1)$ that satisfies the conditions
\begin{equation}\label{sip50}
u\in C^1(\mathbb R^3\setminus \{x_0\})\bigcap\limits_I C^2(\Pi_I) \bigcap C^2(\mathbb R^3\setminus (\overline P\cup \{x_0\})),
\end{equation}
of continuity is a \emph{quasiclassical solution} to the direct scattering problem.
\end{definition}


\subsection{Lippman-Schwinger Integral Equation. Smoothness of Solutions to the Integral Equation.}

Now we are eager to reduce the problem $({\cal P}_1)$ to the Lippman-Schwinger integral equation.

Rewrite the Helmholtz equation in sub-domains $\Pi_I$ as follows
\begin{equation}\label{sip130}
\Delta u(x) + k_0^2 u(x) = (k_0^2 - k_I^2(x))u(x),\;x\in \Pi_I.
\end{equation}
In the bounded region  $\Pi_0=B\setminus\overline P$ one obtains
\begin{equation}\label{sip140}
\Delta u(x) + k_0^2 u(x) = -\delta(x-x_0),\;x\in \Pi_0,
\end{equation}
where $B\supset P$ is a sufficiently large ball centered at zero (denote also $S=\partial B$)  and  $G(x,y) = \frac{\exp(ik_0 |x-y|)}{4\pi |x-y|}$ is the Green function of the Helmholtz equation. Applying the second Green formula one derives
\begin{equation}\label{sip150}
\begin{aligned}
 &\int\limits_{\partial \Pi_I} \left(\frac{\partial u(y)}{\partial \mathbf n}G(x,y) - \frac{\partial G(x,y)}{\partial \mathbf n}u(y) \right)ds_y =  \int\limits_{            \Pi_I} \Bigl(\Delta  u(y)G(x,y)  - \Delta G(x,y)u(y) \Bigr) dy = \\
 &=  \int\limits_{            \Pi_I} \Bigl(-k_I^2(y) u(y)G(x,y)  + k_0^2 G(x,y)u(y) + \delta(x-y)u(y) \Bigr) dy =\\
 &=  u(x) - \int\limits_{   \Pi_I} (k_I^2(y)- k_0^2) G(x,y)u(y) dy,\quad x\in \Pi_I;
 \end{aligned}
\end{equation}
\begin{equation}\label{sip160}
\begin{aligned}
 &\int\limits_{\partial \Pi_J} \left(\frac{\partial u(y)}{\partial \mathbf n}G(x,y) - \frac{\partial G(x,y)}{\partial \mathbf n}u(y) \right)ds_y =  \int\limits_{  \Pi_J} \Bigl(\Delta  u(y)G(x,y)  - \Delta G(x,y)u(y) \Bigr) dy = \\
 &=  \int\limits_{\Pi_J} \Bigl(-k_J^2(y) u(y)G(x,y)  + k_0^2 G(x,y)u(y)  \Bigr) dy =\\
 &=  - \int\limits_{   \Pi_J} (k_J^2(y)- k_0^2) G(x,y)u(y) dy,\quad x\in \Pi_I\;(J\ne I);
 \end{aligned}
\end{equation}
\begin{equation}\label{sip170}
\begin{aligned}
 &\int\limits_{\partial P \cup S} \left(\frac{\partial u(y)}{\partial \mathbf n}G(x,y) - \frac{\partial G(x,y)}{\partial \mathbf n}u(y) \right)ds =  \int\limits_{B\setminus P} \Bigl(\Delta  u(y)G(x,y)  - \Delta G(x,y)u(y) \Bigr) dx = \\
 &=  \int\limits_{B\setminus P} \Bigl(-k_0^2 u(y)G(x,y)  - \delta(y-x_0)G(x,y) + k_0^2 G(x,y)u(y)  \Bigr) dy = -G(x,x_0).
\end{aligned}
\end{equation}

Add equalities  \eqref{sip150}-\eqref{sip170} and take into account the transmission conditions:
\begin{equation}\label{sip180}
\begin{aligned}
 \int\limits_{S} \left(\frac{\partial u(y)}{\partial \mathbf n}G(x,y) - \frac{\partial G(x,y)}{\partial \mathbf n}u(y) \right)ds  =  u(x) &- \sum\limits_J\int\limits_{\Pi_J} (k_J^2(y)- k_0^2) G(x,y)u(y) dy -\\
 & -G(x,x_0),\quad x\in \Pi_I.
\end{aligned}
\end{equation}

Passing in \eqref{sip180} to the limit as the radius of the ball $B$ tends to infinity one gets
\begin{equation}\label{sip190}
  u(x) - \sum\limits_J\int\limits_{\Pi_J} (k_J^2(y)- k_0^2) G(x,y)u(y) dy =  G(x,x_0),\quad x\in \Pi_I.
\end{equation}

The last equation can be rewritten as follows
\begin{equation}\label{sip200}
u(x) - \int\limits_{P}(k^2(y)-k_0^2) G(x,y)u(y)dy = u_0(x),\quad x\in P
\end{equation}
since the definition of the function $k(x).$  Consider also the integral representation  of the total field in the outside of the solid $P:$
\begin{equation}\label{sip210}
u(x) = u_0(x)+ \int\limits_{P} (k^2(y)-k_0^2) G(x,y)u(y)dy,\quad x\in \mathbb R^3\setminus (P\cup \{x_0\}).
\end{equation}

\begin{definition}
The \emph{integral statement} of the direct diffraction problem is understood as the system  $({\cal P}_2)$ consisting of equation \eqref{sip200} in the domain $P$ and representation и \eqref{sip210} outside it.
\end{definition}

The operator in equation \eqref{sip200} is denoted by ${\cal I}-{\cal A}$ and is treated as a mapping in the $L_2(P)$ space.

First, let us show that any solution $u(x)$ of the problem $({\cal P}_2)$ satisfies the smoothness conditions, formulated in the quasiclassical statement of the problem.

\begin{theorem}\label{theor_smooth}
Let equation \eqref{sip200} have a solution $u\in L_2(P).$ Then, the smoothness conditions \eqref{sip50} are satisfied by the total field $u(x),$ extended outside $P$ according to \eqref{sip210}.
\end{theorem}

\begin{proof}

From the definition of the incident wave in the considered statement of the problem it follows that $u_0\in C^{\infty}(\mathbb R^3\setminus \{x_0\}).$

Any solution is infinitely differentiable outside the solid since the smoothness of the integral operator kernel at each  $x\notin\overline P.$

Now consider equation \eqref{sip200}. For each multi-index $I$ one has
\begin{equation}\label{sip220}
u(x)-\int\limits_{\Pi_I} (k^2_I(y)-k_0^2)G(x,y)u(y)dy= \sum\limits_{J\ne I} \int\limits_{\Pi_J} (k^2_J(y)-k_0^2)G(x,y)u(y)dy + u_0(x), \quad x\in \Pi_I.
\end{equation}
The righthand side of the equation is infinitely differentiable in the open domain $\Pi_I$ since, for any $J\ne I,$ one has $G(x,y)\in C^\infty(\Pi_I\times \Pi_J).$

The inclusion $\in L_2(P)$  implies $u\in H^2(P)$ and, consequently, $u\in C^\alpha(\overline P)$ for all $0<\alpha<1/2.$ Then  ${\cal A}u \in C^1(\mathbb R^3\setminus \{x_0\})$ (see.~\cite{b_Vla71}). From the latter follows the inclusion $u\in C^1(\mathbb R^3\setminus \{x_0\})$ which results also in the energy finiteness condition $u\in H^1_{\mathrm{loc}}(\mathbb R^3\setminus\{x_0\}).$

It suffices to prove that $u\in C^2(\Pi_I)$ for any $I.$ Write the following equality for $u$ according to equation \eqref{sip220}:
\begin{equation}\label{sip225}
u(x)=\int\limits_{\Pi_I} (k^2_I(y)-k_0^2)G(x,y)u(y)dy + w(x)= v(x)+w(x), \quad x\in \Pi_I,
\end{equation}
where $w\in C^\infty(\Pi_I).$

Let $x_0\in \Pi_I$ be an arbitrary inner point of the $I$-th sub-domain such that $d=dist(x_0,\partial \Pi_I)>0.$ Introduce the cut-off function  $c\in C^\infty(\Pi_I):$
$$
c(y)=
\begin{cases}
1, & y\in B=B_{d/4}(x_0),\\
0, & y\in B=\Pi_I\setminus B_{3d/4}(x_0).
\end{cases}
$$
Represent $v$ in the form given below:
$$
v(x)=v_1(x)+v_2(x)=\int\limits_{\Pi_I} (k^2_I(y)-k_0^2)G(x,y)u(y) c(y)dy + \int\limits_{\Pi_I\setminus B} (k^2_I(y)-k_0^2)G(x,y)u(y)(1-c(y))dy.
$$
Since the smoothness of the kernel in the second term, one obtains $v_2\in C^\infty (B).$

Note that $u\in C^{0,\alpha}(\Pi_I)$ and $c\in C_0^\infty(\Pi_I),$ то $(k^2_I(y)-k_0^2)G(x,y)u(y) c(y)\in C_0^{0,\alpha}(\mathbb R^3).$ From the latter, using the properties of the volume potential (for details, see \cite{b_ColtonKress_Inv} on page 212), follows the inclusion $v_1\in C^{2,\alpha} (\mathbb R^3).$

Thus, a solution $u(x)$ is twice differentiable in a vicinity of each point  $x_0\in \Pi_I,$ i.e. $u\in C^2(\Pi_I).$

\end{proof}

\subsection{The Equivalency Theorem. Uniqueness of a Solution to the Diffraction Problem}

Let us formulate and prove two important results of investigation of the direct diffraction problem, which are the theorem on equivalency between the differential in the integral formulations of the problem, and the theorem on uniqueness of its quasiclassical solution.

\begin{theorem}\label{theor_equiv}
The problems $({\cal P}_1)$ and $({\cal P}_2)$ are equivalent. More precisely, if $u(x)$ is a quasiclassical solution to the problem $({\cal P}_1)$ then $u$ satisfies equation \eqref{sip200} and representation equality \eqref{sip210}. Vise versa, for any solution $u\in L_2(P)$ to the integral equation \eqref{sip200}, the total field $u(x),$ extended to $\mathbb R^3\setminus \{x_0\}$ by formula \eqref{sip210}, is a quasiclassical solution to the problem $({\cal P}_1)$.
\end{theorem}
\begin{proof}

The former part of the theorem follows from the derivation of the integral equation.

Let $u$ be a solution to equation \eqref{sip200} with $u_0\in C^{\infty}(\mathbb R^3\setminus \{x_0\})$.

The definition of the term  $u_1$ via a volume potential together with smoothness of the term $u_0$ in $\mathbb R^3\setminus \{x_0\}$ imply that $u$ is a solution to the Helmholtz equation in the domains $\Pi_I$ and $\mathbb R^3\setminus (\{x_0\}\cup\overline P).$

The scattered field $u_s(x)=\int\limits_P (k^2(y)-k_0^2)G(x,y)u(y)dy$ satisfies the radiation condition, whereas the transmission conditions are fulfilled since the inclusion $u\in C^1(\mathbb R^3\setminus \{x_0\})$ shown in Theorem \ref{theor_smooth}. Note that the equality $\frac{\partial u}{\partial \mathbf n}  = 0$ can not be written on the edges of the sub-domains $\Pi_I.$
\end{proof}


\begin{theorem}\label{theor_uniqBV}
For any $\Im k(x)\ge 0$ the problem $({\cal P}_1)$ has at most one quasiclassical solution.
\end{theorem}

\begin{proof}

Let us show that the corresponding homogeneous boundary value problem (with $u_0\equiv 0$ in $\mathbb R^3$) formulated for the scattered field $u_s\equiv u$ has only the trivial solution $u_s\equiv 0.$

1. Consider a sufficiently large ball $B \supset \overline P$ of a radius $R.$ Introduce regions  $\Pi_0:=B\setminus(\overline P)$ with the boundary $\partial \Pi_0 = \partial B \cup \partial P,$ and $\Pi_{-1}:={\overline B}^c.$

We shall reduce the original problem for the scattered field $u_s$ to a transmission problem in the domains $\Pi_i.$ For this purpose we denote the restriction of $u_s(x)$ to the closed subsets $\overline \Pi_I$ by $v_I(x).$ The functions $v_I$ satisfy the Helmholtz equation in the corresponding domains:
\begin{equation}\label{sip70}
\begin{array}{ll}
 (\Delta + k_e^2 )v_I(x) = 0,& x\notin P \; (I = -1,0),\\
 (\Delta + k^2_I(x))v_I(x) = 0,& x\in \Pi_I,
 \end{array}
\end{equation}
as well as the transmission conditions at the boundaries of the adjacent domains $\Pi_I, \Pi_J$:
\begin{equation}\label{sip80}
 v_I(x)=v_J(x),\;
 -\frac{\partial v_I(x)}{\partial \mathbf n}= \frac{\partial v_J(x)}{\partial \mathbf n},\quad x \in\partial \Pi'_I.
\end{equation}
The radiation condition is now formulated for the function $v_{-1}:$
\begin{equation}\label{sip90}
 \frac{\partial v_{-1}}{\partial r} = i k_0 v_{-1} + o\left( {\frac{1}{r}} \right), \; (\Im k_0 = 0);\quad  v_{-1}(r)= O\left(\frac{1}{r^2}\right),\;(\Im k_0 > 0).
  \end{equation}

The first Green formula, applied to the functions $\overline v_I, v_I$ in the bounded domains  $\Pi_0$  and $\Pi_I\subset P,$  as well as the Helmholtz equation yield in the following relation:
\begin{equation}\label{sip100}
\begin{aligned}
\int\limits_{\Pi_I}\bigl(\overline v_I\triangle v_I+|\nabla v_I|^2\bigr)dx &=-\int\limits_{\Pi_I} k^2_I|v_1|^2 dx+\\
&+\int\limits_{\Pi_I} |\nabla v_I|^2 dx = \int\limits_{\partial \Pi_I} \overline v_I \frac{\partial v_I}{\partial\mathbf n} ds,\\
\int\limits_{\Pi_0} \bigl(\overline v_0\triangle v_0 +|\nabla v_0|^2\bigr)dx&=-k^2_0\int\limits_{\Pi_0} |v_0|^2dx +\\
&+\int\limits_{\Pi_0} |\nabla v_0|^2 dx = \int\limits_{\partial \Pi_0} \overline v_0 \frac{\partial v_0}{\partial\mathbf n} ds,
\end{aligned}
\end{equation}

Add equalities \eqref{sip100}, using the transmission conditions \eqref{sip80}:
\begin{equation}\label{sip110}
\begin{aligned}
\int\limits_{\partial B } \overline v_0 v_{0,\mathbf n} ds=
& -\sum\limits_I\int\limits_{\Pi_I} k^2_I |v_1|^2 dx-k^2_0\int\limits_{\Pi_0} |v_0|^2 dx+ \\
&+\int\limits_{V_0} |\nabla v_0|^2 dx + \sum\limits_{I} \int\limits_{\Pi_I} |\nabla v_I|^2 dx
 = - \int\limits_{\partial B  } \overline v_{-1} \frac{\partial v_{-1}}{\partial\mathbf n} ds.
\end{aligned}
\end{equation}

Consider the imaginary part of the latter relation and take into account the radiation conditions:
$$
\begin{aligned}
\Im \left(\int\limits_{\partial B} \overline v_{-1} \frac{\partial v_{-1}}{\partial\mathbf n} ds\right)&=
\Im \left(\ \int\limits_{\partial B} (i k_0 v_{-1} + o(R^{-1}))\overline v_{-1} ds\right)=\\
    &=k_0\int\limits_{\partial B} |v_{-1}|^2 ds +
         \int\limits_{\partial B}  o(R^{-2})ds=\\
    &=k_0\int\limits_{\partial B} |v_{-1}|^2 ds +o(1)=0.
\end{aligned}
$$
Application of the Rellich lemma (see \cite{b_ColKre87ru} on p.88) results in equality $v_{-1}(x)\equiv 0$ for each point  $x\in\Pi_{-1}.$

3. Show now that the relation $v_{-1}(x)\equiv 0$ holds at the points of the inhomogeneity domain.

Consider an arbitrary <<external>> parallelepiped $\Pi_I$ such that $\partial \Pi_I\cap\partial P = S \ne\varnothing.$ Represent the solution $u$ as follows (see also the previous subsection):
\begin{equation}\label{sip111}
\begin{aligned}
u(x) & =\int\limits_P G(x,y)(k^2(y)-k_0)u(y)dy =\\
     &=\sum\limits_{J\ne I}\int\limits_{\Pi_J} G(x,y)(k^2_J(y)-k_0)u(y)dy + \int\limits_{\Pi_I} G(x,y)(k^2_I(y)-k_0)u(y)dy=\\
     &=v(x)+w(x),\quad x\in P.
\end{aligned}
\end{equation}

It can be shown that the function $u(x)$ is infinitely differentiable in a sufficiently small vicinity $U$ of an arbitrary point $x_0\in S.$ Consider the following representation of $w(x):$
\begin{equation}\label{sip112}
\begin{aligned}
w(x)  &=\int\limits_{\Pi_I} G(x,y)(k^2(y)-k_0)u(y) c(y)dy +\\
      &+\int\limits_{\Pi_I\setminus U'} G(x,y)(k^2(y)-k_0)u(y) (1-c(y))dy = w_1(x) + w_2(x),
\end{aligned}
\end{equation}
where $c(y)\in C^{0,\alpha}_0(U)$ is a compactly supported cut-off function such that $c(y)\equiv 1$ in $B_r(x_0)=U'\subset U.$ Then the smoothness of the integral operator's kernel implies $w_2\in C^\infty (U').$ Further,  the inclusion $w_1\in C^2(\mathbb R^3)$  is valid since the term $w_1$ is a Newtonian potential with the compactly supported smooth density $(k^2(y)-k_0)u(y) c(y)\in C_0^{0,\alpha}(\mathbb R^3)$ (see \cite{b_ColtonKress_Inv} on p.207).

Thus, the function $u\in C^2(U')$ is a solution to the Helmholtz equation such that $u\equiv 0$ in the sub-domain $U'\setminus P.$ From the unique continuity principle (\cite{b_ColtonKress_Inv}, p.212) it follows now that $u\equiv 0$ in the domain $U'$ and, consequently in the entire parallelepiped $\Pi_I.$

Similarly, one can consider all sub-domains $\Pi_I,$ repeat the above arguments, and conclude that $u\equiv 0$ in $P.$

If the condition $\Im k>0$ holds in the entire space, then, from the second relation in \eqref{sip90}  it follows that $u(x)=O(R^{-2})$ an the sphere $\partial B.$ Consequently, the left-hand  side of equality \eqref{sip110} vanishes as $R\to +\infty.$ As a results, one obtains
\begin{equation}\label{sip120}
\begin{aligned}
&-\sum\limits_I\int\limits_{\Pi_I} k^2_I|v_I|^2 dx - k^2_e\int\limits_{\Pi_0} |v_0|^2dx + \\
  &+\sum\limits_{I} \int\limits_{\Pi_I} |\nabla v_I|^2 dx  \to 0,\, R\to +\infty.
\end{aligned}
\end{equation}
For the imaginary part of \eqref{sip120} holds the following relation:
$$
\sum\limits_I \Re k_I\cdot\Im k_I\int\limits_{\Pi_I} |v_I|^2 dx + \Re k_0\cdot\Im k_0 \int\limits_{\Pi_0} |v_0|^2 dx \to 0
$$
as $R\to +\infty.$ Both terms in the latter expression are of the same sign due to the properties of the medium. From that we conclude that $v_I(x)\equiv 0,\;x\in \Pi_I.$

If $k_0>0$ outside the inhomogeneity domain and $\;\Re k\!\cdot\!\Im k(x)>0$ inside it, then we similarly deduce that $u\equiv 0$ in $P:$ using the Rellich lemma we first get $v_{-1}\equiv 0,$ and then, as in the item 3. of the current proof, we obtain that the solution is trivial inside the solid $P.$
\end{proof}

The next statement on results from Theorems \ref{theor_equiv} and \ref{theor_uniqBV}:
\begin{theorem}
The operator
$$
({\cal I}-{\cal A}):L_2(P)\to L_2(P)
$$
is continuously invertible
\end{theorem}
\begin{proof}

For any $u\in L_2(P)$ one has ${\cal A}u\in H^2(P).$ From the latter it follows that ${\cal A}:L_2(P)\to L_2(P)$ is a compact operator.

Let $u_0\equiv 0$ in $\mathbb R^3.$ Then the boundary value problem $({\cal P}_1)$ has only the trivial solution (see Theorem \ref{theor_uniqBV}). Thus, due to the equivalency between $({\cal P}_1)$ and $({\cal P}_2),$ $u=0$ is the only solution to the integral equation $({\cal I}-{\cal A})u=0.$

Thus, $({\cal I}-{\cal A}):L_2(P)\to L_2(P)$ is an injective Fredholm operator with index zero.
\end{proof}

\section{The Inverse Problem of Reconstructing the Refractive Index}
\subsection{Statement of the Inverse Problem}

Consider in $\mathbb R^3$ an inhomogeneous parallelepiped $P$ characterized by an unknown refractive index $n(x).$
Assume, as in the statement of the direct diffraction problem, that $n(x)$ and $k(x)=n(x)k_0$ are piecewise-H\"{o}lder functions in the domain $P$ with a given mesh and a set of sub-domains $\Pi_I.$

Introduce a bounded domain $D$ such that $\overline D\cap\overline P=\varnothing,$ and assume that in the points $x\in D$ we are given the known values of the total field at a fixed frequency $\omega:$
\begin{equation}\label{sip300}
U(x,t)=U_0(x,t)+U_s(x,t),\; U_s(x,t)=u_se^{-i\omega t}.
\end{equation}

The monochromatic incident wave $U_0(x,t)=u_0(x)e^{-i\omega t}$ is defined according to \eqref{sip35}, and the source of the field is located in an arbitrary point $x_0\notin \overline P\cup\overline D.$

In the proposed statement of the inverse diffraction problem we use the system $({\cal P}_2)$ of integral equalities which represent the relation between the total field $u(x)$ and the function $k(x)$ (it is shown above that formulations $({\cal P}_1)$ and $({\cal P}_2)$ are equivalent).

We are eager to reconstruct the function $k(x)$ in  the parallelepiped $P$ using measurements of the total field $u(x)$ at points of the bounded domain  $D:$
\begin{equation}\label{sip303}
\int\limits_{P} (k^2(y)-k_0^2) G(x,y)u(y)dy = u(x) - u_0(x) ,\quad x\in D
\end{equation}
and taking into account the equation
\begin{equation}\label{sip306}
u(x) - \int\limits_{P}(k^2(y)-k_0^2) G(x,y)u(y)dy = u_0(x),\quad x\in P
\end{equation}
in the inhomogeneity domain $P.$
\subsection{The Two-step Method for Solving the Inverse Diffraction Problem}

In the domain $P,$ we introduce the function
$$J(x)=(k^2(y)-k_0^2)u(x),$$
assuming that the condition $|k(x)|\ge \tilde k>k_0$ holds everywhere in $P.$
From the representation of the total field in the outside of the solid $P$ follows the equation
\begin{equation}\label{sip310}
\int\limits_{P} G(x,y) J(y)dy= u(x)-u_0(x)=u_s(x), \quad x\in D,
\end{equation}
for determination of $J(x),$ where as equation \eqref{sip200} can be rewritten as below:
\begin{equation}\label{sip320}
\frac{J(x)}{k^2(x)-k_0^2} -\int\limits_{P} G(x,y) J(y)dy= u_0(x), \quad x\in P.
\end{equation}

The idea of the proposed two-step method for reconstruction of the unknown coefficient $n(x)$  is as follows:
\begin{itemize}
\item Given the known values of the incident wave $u_0(x)$  and the total field $u(u)$ in the domain $D,$ we find, in the domain $P,$ the solution $J$ to equation  \eqref{sip310}.
\item We reconstruct the function $k(x)$ inside $P$ using relation  \eqref{sip320}.
\end{itemize}

\subsection{On Non-uniqueness of a Solution to the Integral Equation}
Let us show that the homogeneous integral equation \eqref{sip310} has non-trivial solutions for any $k_0.$ For example we give below an argumentation for the case of a cubic solid $P=[-1;1]^3\subset R^3.$

Consider the function $\psi(x)=(1-x_1^2)^2(1-x_2^2)^2(1-x_2^2)^2$ and introduce $J(x)=-(\triangle + k_0^2)\psi.$ As  $\psi$ satisfies the homogeneous boundary conditions $\psi\bigr|_{\partial P}=\frac{\partial \psi}{\partial\mathbf n}\bigr|_{\partial P'}=0,$ then the representation
$$
\psi(x)=\int\limits_P G(x,y)J(y)dy
$$
can be given. Introduce the potential $v(x)=\int\limits_P G(x,y) J(y)dy,\quad x\in \mathbb R^3.$ Then at any point $x\in P$ one obtains $v(x)=\psi(x).$ However, the relation $v\equiv 0$ holds outside the closed cube $\overline P:$
$$
\begin{aligned}
0&=\int\limits_{\partial P} \Bigl(\psi(y)\frac{\partial G(x,y)}{\partial\mathbf n} - G(x,y)\frac{\partial \psi(y)}{\partial\mathbf n} \Bigr) ds_y=
   \int\limits_{P} \Bigl(\psi(y) \triangle_y G(x,y)- G(x,y)\triangle_y \psi(y) \Bigr)dy=\\
 &=\int\limits_{P} \Bigl(-k_0^2\psi(y) G(x,y)- G(x,y)\triangle_y \psi(y) \Bigr)dy = \int\limits_{P} G(x,y) J(y) dy=v(x),\quad x\notin\overline P.
\end{aligned}
$$
The similar result can be obtained in the case of a domain  $P$ of an arbitrary shape. For that purpose, define $\psi$ as an arbitrary smooth function with a compact support in $P$ so as to satisfy the conditions $\psi\bigr|_{\partial P}=\frac{\partial \psi}{\partial\mathbf n}\bigr|_{\partial P'}=0$ and repeat the above analysis.


\subsection{On Uniqueness of a Piecewise-constant Solution $J(x)$ }

It is shown above that the integral equation of the first kind has an infinite set of smooth solution.

We prove below that a unique solution  $J$ can be obtained in the class of piecewise constant functions. Note that for approximate solving equation \eqref{sip310} we are going to apply the collocation method. That is why the choice of piecewise constant functions $J$ is a reasonable one. Thus, we find  $J$ in the following way:
\begin{equation}\label{sip250}
J(x)=\sum\limits_I J_I\chi_I(x),
\end{equation}
where $J_I\in\mathbb C$  are the unknown coefficients, and  $\chi_I(x)$ are the indicator functions of the sub-domains  $\Pi_I$ (at points $x\in\Pi_I$ the function $J(x)$ can be defined by any constant value). Note that the class of piecewise constant solutions is sufficient for solving applied problems of physics, medical tomography etc. (see the introduction of the paper).

Below, we formulate and prove the theorem on uniqueness of a piecewise constant solution $J(x)$ to equation \eqref{sip320}.

\begin{theorem}\label{theor300}\label{theor_uniqIE}
 Consider a fixed set of $n^3$ rectangular sub-domains $\Pi_I$ in the inhomogeneity domain $P.$ Let
  \begin{equation}\label{sip330_0}
k_0 > \frac{\pi^2n^3}{2l},\quad l=\min\limits_i|b_i-a_i|.
  \end{equation}
If equation
  \begin{equation}\label{sip330}
  \int\limits_P G(x,y) J(y)dy = u_s(x),\quad x\in D,\;\overline D\cap\overline P=\varnothing,\; u_s\in C^\infty(\overline D)
  \end{equation}
has a piecewise constant solution $J(x),$ then this solution is the unique one.
\end{theorem}

\begin{proof}
1. Consider the homogeneous equation
\begin{equation}\label{sip340}
  \int\limits_P G(x,y) J(y)dy = 0,\quad x\in D.
  \end{equation}

Introduce the volume potential
\begin{equation}\label{sip350}
v(x)=\int\limits_P G(x,y) J(y)dy,\quad x\in \mathbb R^3.
\end{equation}

Since $J(x)$ is a piecewise constant function then $v(x)\in C^1(\mathbb R^3).$ As a result, the transmission conditions
\begin{equation}\label{sip360}
v|_{\partial \Pi_I}=\frac{\partial v}{\partial\mathbf n}\bigr|_{\partial \Pi_I'}=0
\end{equation}
hold on the boundaries of the parallelepipeds $\Pi_I.$

Moreover, the inclusions $v\in C^2(\Pi_I)$ are valid. Consequently, the Helmholtz equation
\begin{equation}\label{sip370}
(\triangle + k_0^2)v(x)= -J_I,\quad x\in \Pi_I
\end{equation}
holds in the inner points  $x\in\Pi_I$ (in the classical sense).

Outside $\overline P,$ we obtain $(\triangle + k_0^2)v(x)= 0$ and $v\in C^\infty (\mathbb R^3\setminus\overline P).$

By the assumption of the theorem, the function $v$ is equal to zero in the domain $D\subset \mathbb R^3\setminus\overline P.$ Then, applying the unique continuation principle (\cite{b_ColtonKress_Inv}, p.212), we derive that о $v\equiv 0$ everywhere in $\mathbb R^3\setminus\overline P.$

From the inclusion $v(x)\in C^1(\mathbb R^3)$ follows the relation
\begin{equation}\label{sip380}
v|_{\partial P}=\frac{\partial v}{\partial\mathbf n}\bigr|_{\partial P'}=0.
\end{equation}

\vspace{0.5em}

2.
Introduce the fundamental solution $\bar G(x,y)=\frac{e^{-ik_0|x-y|}}{4\pi |x-y|}$ to the Helmholtz equation. Apply the second Green formula to the functions $v, \bar G$ in the domains $\Pi_I,$ take into account the homogeneous boundary conditions \eqref{sip380} and the transmission conditions on $\partial \Pi_I$ for the function $v:$
\begin{equation}\label{sip390}
\begin{aligned}
0&=\int\limits_{\partial P} \Bigl(  v(y)\frac{\partial}{\partial\mathbf n}\bar G(x,y)  - \bar G(x,y) \frac{\partial}{\partial\mathbf n}v(y)\Bigr) ds_y =\\
 &=\sum\limits_I \int\limits_{\partial \Pi_I}\Bigl(v(y)\frac{\partial}{\partial\mathbf n}\bar G(x,y)-\bar G(x,y) \frac{\partial}{\partial\mathbf n}v(y)\Bigr)ds_y=\\
 &=\sum\limits_I \int\limits_{\Pi_I}\Bigl(v(y)\triangle_y\bar G(x,y)-\bar G(x,y) \triangle_y v(y)\Bigr) dy=\\
 &=\sum\limits_{I\ne I_0} \int\limits_{\Pi_I}\Bigl(-k_0^2 v(y)\bar G(x,y)-\bar G(x,y) \triangle_y v(y)\Bigr)dy + \\
 &+\int\limits_{\Pi_{I_0}}\Bigl(-k_0^2 v(y)\bar G(x,y) - \delta(x-y)v(y)-\bar G(x,y) \triangle_y v(y)\Bigr)dy=\\
 &=-v(x) + \sum\limits_I J_I\int\limits_{\Pi_I} \bar G(x,y)dy,\quad x\in \Pi_{I_0}.
\end{aligned}
\end{equation}
From the latter follows the equality
\begin{equation}\label{sip400}
v(x) = \int\limits_{P} \bar G(x,y) J(y)dy,\quad x\in P.
\end{equation}

Subtracting \eqref{sip400} from equation \eqref{sip350} one deduces
$$
w(x) = \int\limits_P \frac{\sin(k_0|x-y|)}{4\pi |x-y|} J(y)dy =\int\limits_P G_0(|x-y|) J(y)dy  \equiv 0,\quad x\in P.
$$

As the potential $v$ so the function $w$ satisfies the transmission conditions on $\partial P$ which implies
\begin{equation}\label{sip410}
w|_{\partial P}=\frac{\partial w}{\partial\mathbf n}\bigr|_{\partial P'}=0.
\end{equation}
In addition, the function $w\in C^\infty(\mathbb R^3\setminus\overline P)$ satisfies outside $\overline P$ the equation $(\triangle + k_0^2)w=0$ and the Sommerfeld radiation conditions. Consequently, $w\equiv 0$ in $\mathbb R^3\setminus\overline P.$

\vspace{1em}
3. Thus, $w\equiv 0$ in $\mathbb R^3.$ Then for the Fourier transform ${\cal F}w(\xi)$ holds the similar relation ${\cal F}w\equiv 0$ everywhere in $\mathbb R^3.$


Introduce the grid parameters
$$
h_1=(b_1-a_1)/n_1, \quad h_2=(b_2-a_2)/n_2, \quad h_3=(b_3-a_3)/n_3
$$
and the parallelepiped
$$
\Pi_0=[a_1,a_1+h_1]\times[a_2,a_2+h_2]\times[a_3,a_3+h_3].
$$
Then all other finite elements can be defined via shifts of the sub-domain  $\Pi_0$ by appropriate vectors:
\begin{equation}\label{sip430}
\Pi_I=\Pi_{i_1i_2i_3}=\Pi_0 + r_{i_1i_2i_3}, \mbox{ где } r_{i_1i_2i_3}=r_I=(i_1h_1,i_2h_2,i_3h_3).\quad 0\le i_k <n.
\end{equation}

The function $w(x)$ can now be represented as follows
\begin{equation}\label{sip440}
w(x)=\sum\limits_I J_I\int\limits_{\Pi_0+r_I} G_0(|x-y|) dx= \sum\limits_I J_I\int\limits_{\Pi_0} G_0(|x-y-r_I|)dx.
\end{equation}

Evaluate the Fourier transform of $w$ taking in to account the latter relation:
\begin{equation}\label{sip450}
\begin{aligned}
{\cal F}w&= \sum\limits_I J_I   {\cal F}\Bigl( G_0(|x-r_I|)\ast \chi_0(x)\Bigr)=\\
         &=\sum\limits_I J_I   {\cal F}\Bigl( G_0(|x-r_I|)\Bigr) {\cal F}\Bigl(\chi_0(x)\Bigr) =
         {\cal F}\chi_0(\xi) {\cal F}G_0(\xi) \sum\limits_I J_I e^{ir_I\cdot\xi}=\\
         & =    (2\pi)^{-3}\prod\limits_{k=1}^3\frac{e^{-ih_k\xi_k}-1}{\xi_k}\cdot \bigl(\delta(|\xi|-k_0)\bigr)  \cdot\sum\limits_I J_I e^{ir_I\cdot\xi}.
\end{aligned}
\end{equation}

The relation ${\cal F}w\equiv 0$ is reduced to the equality
\begin{equation}\label{sip460}
\sum\limits_I J_I e^{ir_I\cdot\xi} \equiv 0.
\end{equation}
on the centered sphere $S_{k_0}$ of the radius $k_0.$

By $S^{n-1}$ we denote a unit sphere in $\mathbb R^n$ ($n\ge 2$).

\vspace{1em}
4. Let us shown that the functions $e^{ir_I\cdot\xi}$ are linearly independent on the sphere $S_{k_0}.$  To this end, we shall prove that the corresponding Gram matrix $\Gamma$ is nonsingular.

For an arbitrary matrix element $\Gamma_{II'}$ one deduces
\begin{equation}\label{sip470}
\begin{aligned}
\Gamma_{II'}&=\int\limits_{S_{k_0}} e^{ir_I\cdot\xi}e^{-ir_{I'}\cdot\xi} ds_\xi
            =k_0^2\int\limits_{S^2} e^{ik_0(r_I-r_{I'})\cdot\xi} ds_\xi
            =k_0^2\int\limits_{S^2} e^{ik_0 r_{II'}\cdot\xi} ds_\xi =\\
            &=k_0^2\int\limits_{S^2} e^{ik_0 |r_{II'}| \omega_{II'}\cdot\xi} ds_\xi
            =2\pi k_0^2 \int\limits_{-1}^1 e^{ik_0|r_{II'}|t}dt.
\end{aligned}
\end{equation}
In the above evaluation that the integrals  $\int\limits_{S^{n-1}}f(\omega\cdot\xi)ds_\xi$ over the unit centered sphere $S^{n-1}$ do not depend on the variable $\omega\in S^{n-1}$ and, consequently, can be presented  \cite{b_Natterer} as below:
\begin{equation}\label{sip480}
\int\limits_{S^{n-1}}f(\omega\cdot\xi)ds_\xi = |S^{n-2}|\int\limits_{-1}^1 f(t) (1-t^2)^{(n-3)/2}dt.
\end{equation}
From \eqref{sip470} it follows that
\begin{equation}\label{sip490}
\Gamma_{II'}=
\begin{cases}
4\pi k_0 \frac{\sin(k_0|r_{II'}|)}{|r_{II'}|},  &I\ne I',\\
4\pi k_0^2,                             &I =I'.
\end{cases}
\end{equation}

\vspace{1em}
5. Represent $\Gamma$ via sum
$$
\Gamma = 4\pi k_0 (k_0 \tilde I + \tilde\Gamma),
$$
where $\tilde I$ is the unit matrix, and obtain the estimate
$$
\|\tilde \Gamma\|_\infty = \max\limits_{I} \sum\limits_{I'} |\tilde \Gamma_{II'}| \le \frac{\pi^2n^3}{2l}.
$$
The latter implies that the determinant of the Gram matrix is nonzero, since its diagonal element dominate at sufficiently large $n.$

Fix the row index $I'=(0,0,0)$ and define $h=\min\{h_1,h_2,h_3\}$:
\begin{equation}\label{gc473}
\begin{aligned}
\sum\limits_{I\ne I'}|\tilde \Gamma_{II'}|
    &=\sum\limits_{I\ne I'}\frac{1}{|r_{II'}|} = \sum\limits_{(i_1,i_2,i_3)=(0,0,1)}^{(n,n,n)}\frac{1}{\sqrt{(i_1h_1)^2+(i_2h_2)^2+(i_3h_3)^2}}\le\\
    &\le \frac{6}{h}+\frac{1}{h}\sum\limits_{(i_1,i_2,i_3)=(1,1,1)}^{(n,n,n)}\frac{1}{\sqrt{(i_1)^2+(i_2)^2+(i_3)^2}}
     \le \frac{1}{h}\Bigl(6 +2 \iiint\limits_1^{\quad n}\frac{dx}{|x|}\Bigr)\le\\
    &\le \frac{1}{h}\Bigl(6 +0.5 \iiint\limits_{1<|x|<n}\frac{dx}{|x|}\Bigr)
    =  \frac{1}{h}\Bigl(6 + \pi^2 \int\limits_1^n\frac{r^2}{r}dr\Bigr) < \frac{\pi^2}{2} \frac{n^2}{h}\le \frac{\pi^2n^3}{2l}.
\end{aligned}
\end{equation}

The proof is complete.
\end{proof}

Since equation \eqref{sip320} represents the relation between $J,$ $k$ and $u,$ it now follows from theorem \ref{theor_uniqIE} that the solution $k(x)$ to the inverse diffraction problem, corresponding to the piecewise constant $J,$ is also unique.

\begin{remark} [on existence of a solution]
Let us assume that the righthand side of equation \eqref{sip330} is an element of a linear span of functions $\int_{\Pi_I} G(x,y)dy$ (for a given mesh on $p$). Then the operator of the lefthand side of equation \eqref{sip330} can be treated as mapping in finite-dimensional spaces. Such mapping are continuously invertible which results from theorem \ref{theor_uniqIE}. In addition, the solutions to the inverse problem depend continuously on the given data.
\end{remark}

\subsection{Numerical Solution of the Inverse Problem}

To solve equation \eqref{sip310} numerically, we apply the collocation method.

The current $J$ is sought in the form of a linear combination $\sum\limits_{j=1}^N c_j v_j(x)$ of piecewise constant basis functions. The collocation points are defined as follows: $r_i\in D$ ($i=1,\ldots,N.$)

\vspace{0.7em}

Below we describe the conditions and results of computational experiments.

We consider the cubic inhomogeneity region $P$ with the edge length of $0.15 m.$ The area of such a size is quite consistent with the objects studied, e.g., in breast cancer diagnosis.


There was carried out a series of computational experiments under various conditions. First, the array of receivers (points for measuring the total field) was defined on planes parallel to $xy-$ or $xz-$ planes. Second, the distance from the nearest receiver to the scatterer was varied (this parameter is denoted by $d_r$). Third, the calculations were made with or without artificially introduced errors (noise). The distance $d_s $ from the source of the incident wave $ u_0 $ to the body $P$ is fixed: $ d_s = 0.003  m.$


All the figures below graphically represent the values of the real (under the letter (a)) and imaginary (under the letter (b)) parts of the function $k(x).$


We consider the following sample problem (SP): the inhomogeneity domain $P$ is characterized by a given  complex-valued function $k(x).$ The graphical representation of the real and imaginary parts of the function $k(x)$ is in the figures \ref{fig}(a) and \ref{fig}(b), respectively.

\begin{figure}
\begin{center}
\begin{tabular}{cc}
\resizebox*{6cm}{!}{\includegraphics{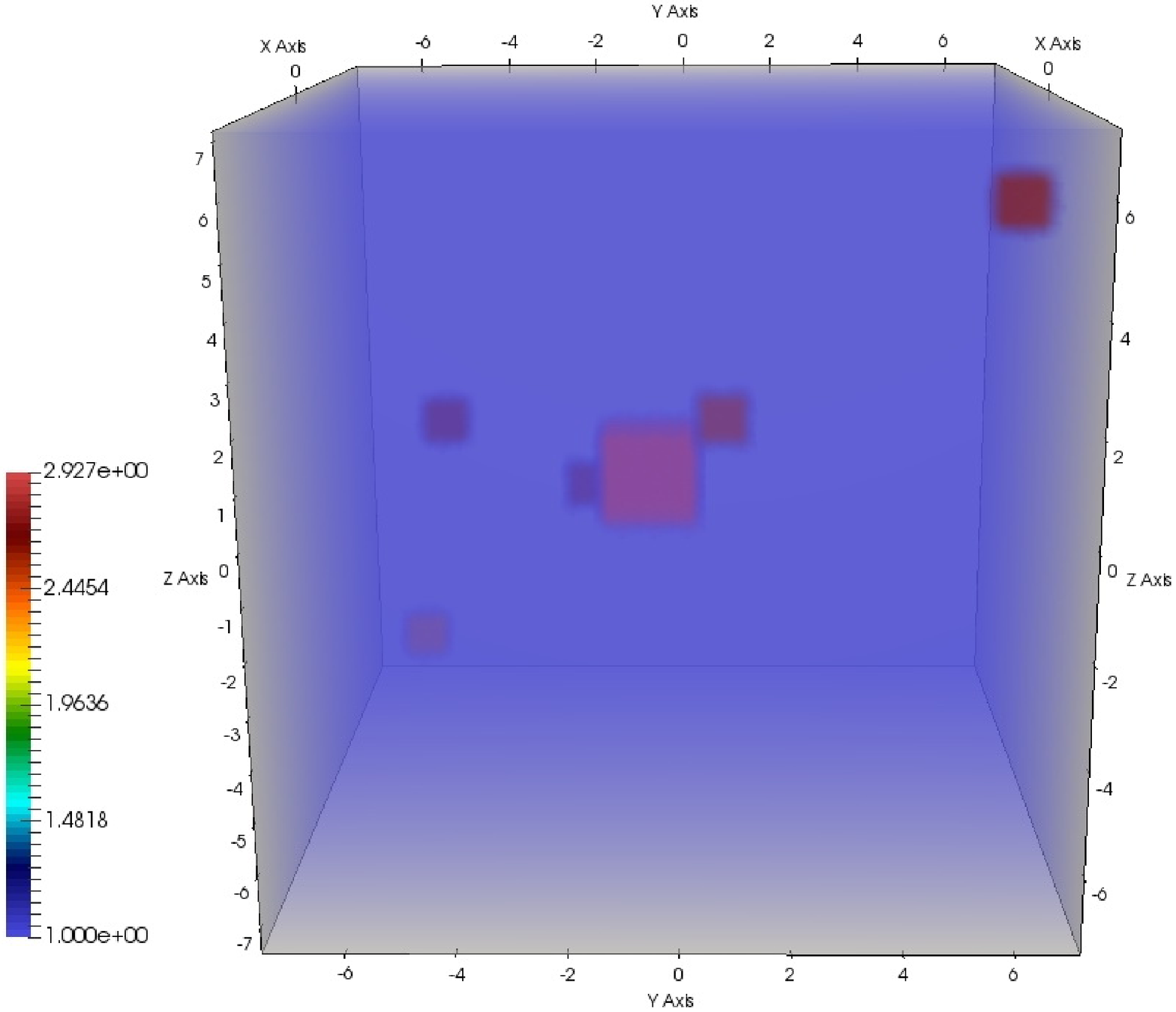}} & \resizebox*{6cm}{!}{\includegraphics{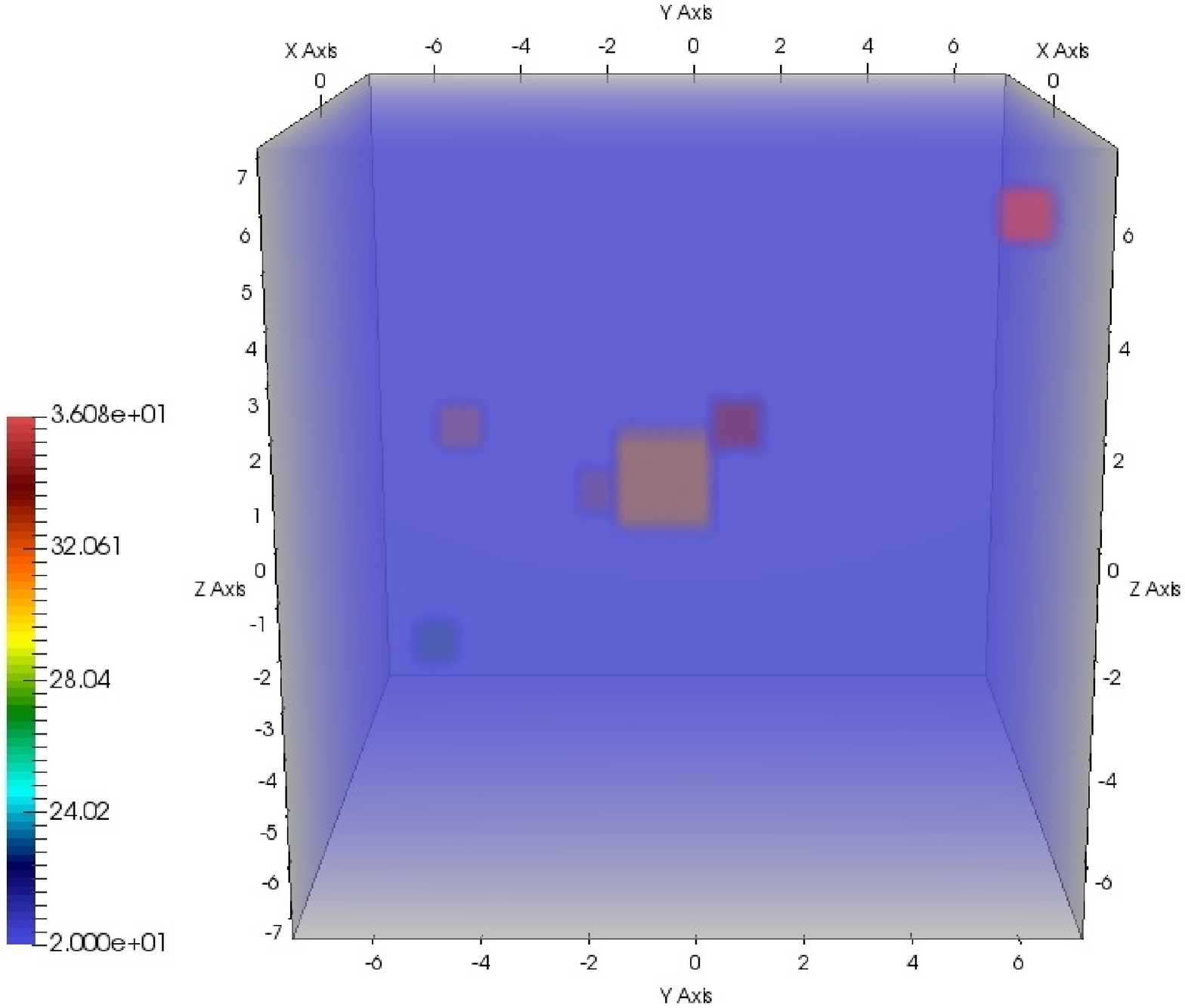}}\\
(a) & (b)
\end{tabular}
\caption{Exact solution of SP: the real  (a) and the imaginary (b) parts of the function $k(x)$.}
\label{fig}
\end{center}
\end{figure}


Given the function $k (x),$ we first solve  the direct diffraction problem. As a result, the field $u(x)$ and the "current" $J(x)$ are determined in the region $P.$ These functions are then used for modeling the total field of the inverse problem.


Finally, the solution to the inverse is found using the two-step approximate method. The functions $\Re k(x) $ and $\Im k(x),$ shown in figures \ref{fig1}(a) and \ref{fig1}(b), correspond to the exact input data (no noise is added). This data is the values of the scattered field $u_s=u - u_0$ at the receivers points. Here $u_0$ is the given incident wave and $u$ is the solution to the direct problem of diffraction.

\begin{figure}
\begin{center}
\begin{tabular}{cc}
\resizebox*{6cm}{!}{\includegraphics{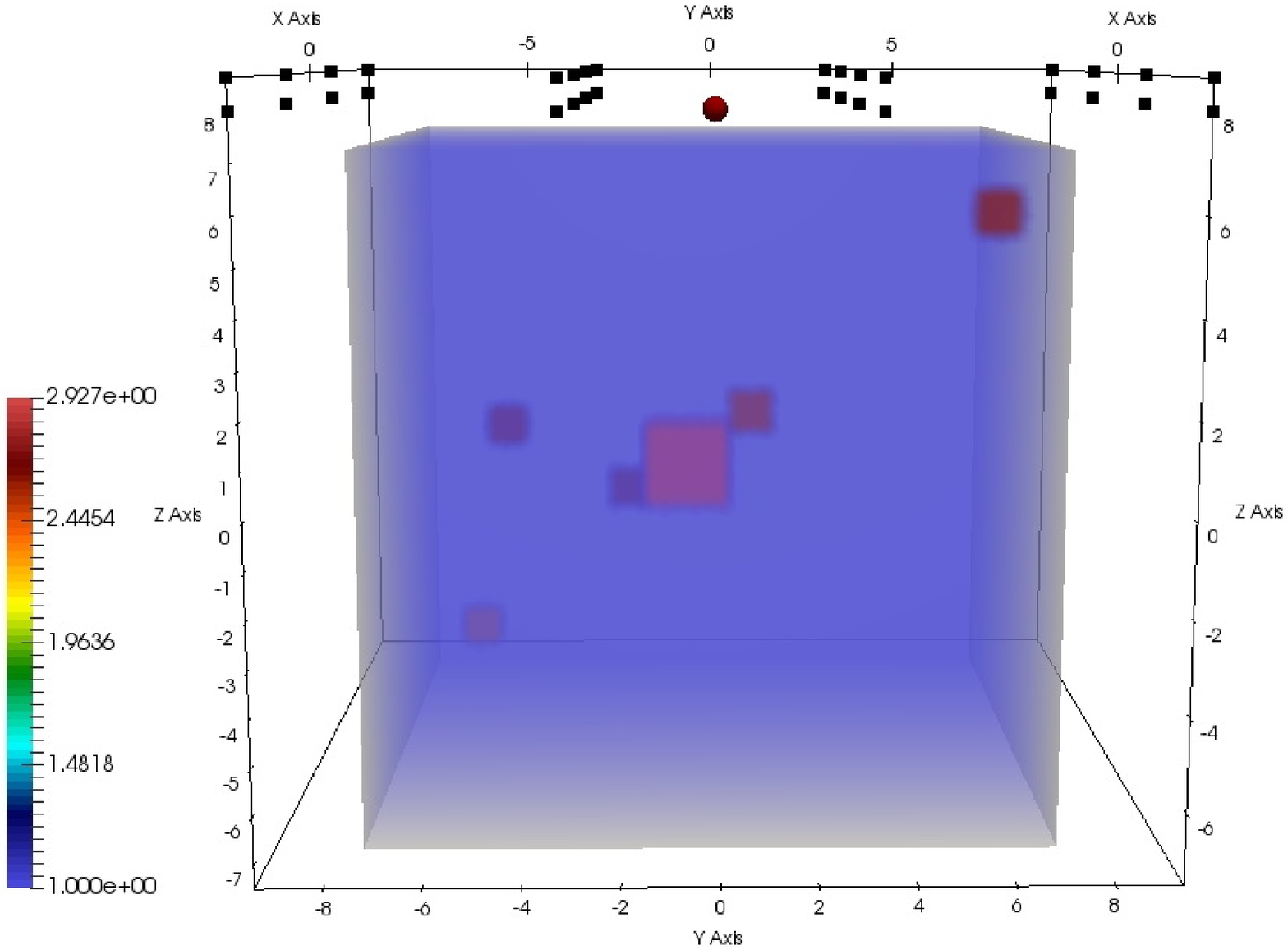}} & \resizebox*{6cm}{!}{\includegraphics{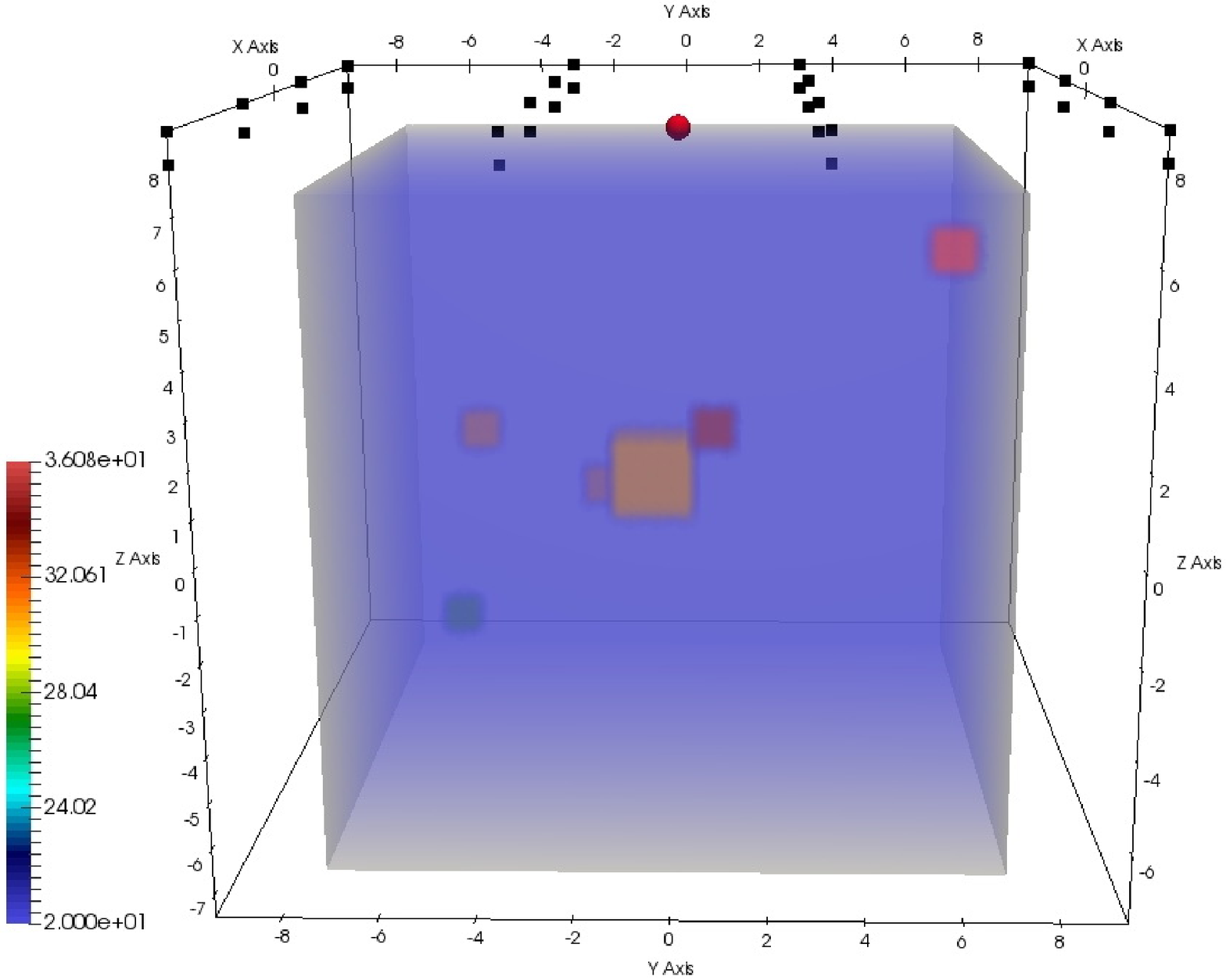}}\\
(a) & (b)
\end{tabular}
\caption{Approximate solution of SP: $d_r=0.005 m$, the receivers are in planes parallel to $xy-$plane, no noise is added.}
\label{fig1}
\end{center}
\end{figure}


In the next experiment (see Figures \ref{fig2}(a) and \ref{fig2}(b)), the distance from the body to the receivers is increased up to $0.05 m.$ In this case, the approximate solution significantly differs from the exact one. One can see numerous "artifacts" in the figures.
\begin{figure}
\begin{center}
\begin{tabular}{cc}
\resizebox*{6cm}{!}{\includegraphics{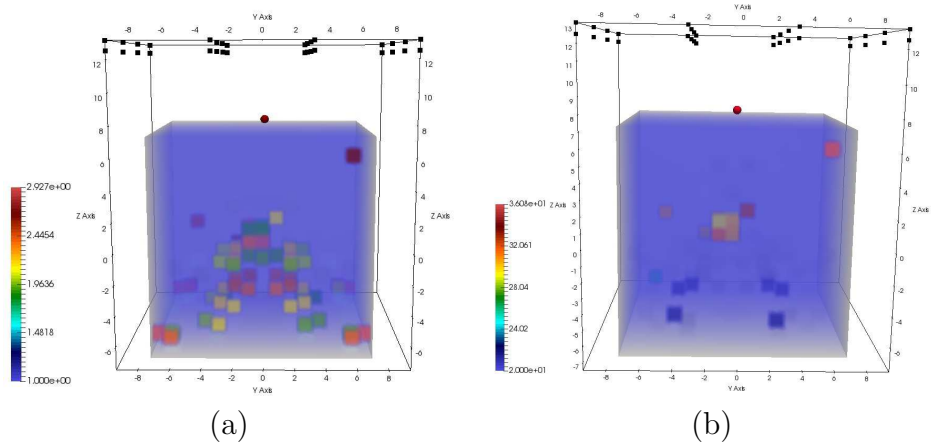}} & \resizebox*{6cm}{!}{\includegraphics{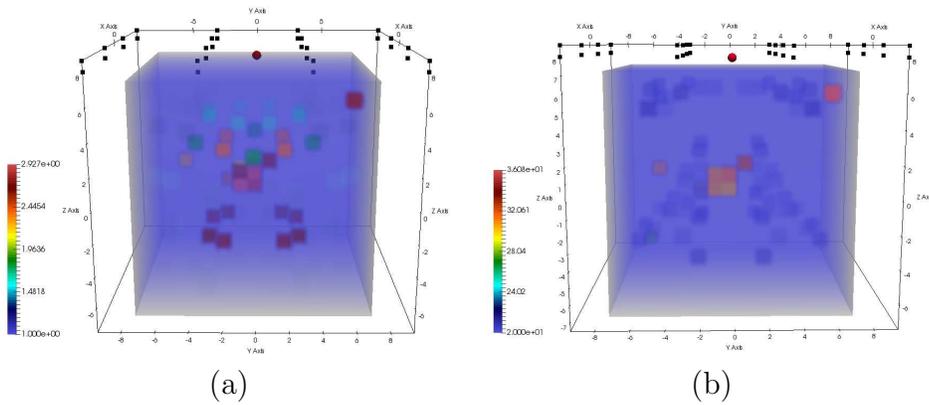}}\\
(a) & (b)
\end{tabular}
\caption{Approximate solution of SP: $d_r=0.05 m$, the receivers are in planes parallel to $xy-$plane, the input data is noisy.}
\label{fig2}
\end{center}
\end{figure}


In the third experiment (see Figures \ref{fig3}(a) and \ref{fig3}(b)), the distance $d_r$ is reduced to $0.005 m,$ whereas the data is noisy. Here again, the solution is quite inaccurate and shows the "artifacts".  In addition, there is also a loss of "true" heterogeneities.
\begin{figure}
\begin{center}
\begin{tabular}{cc}
\resizebox*{6cm}{!}{\includegraphics{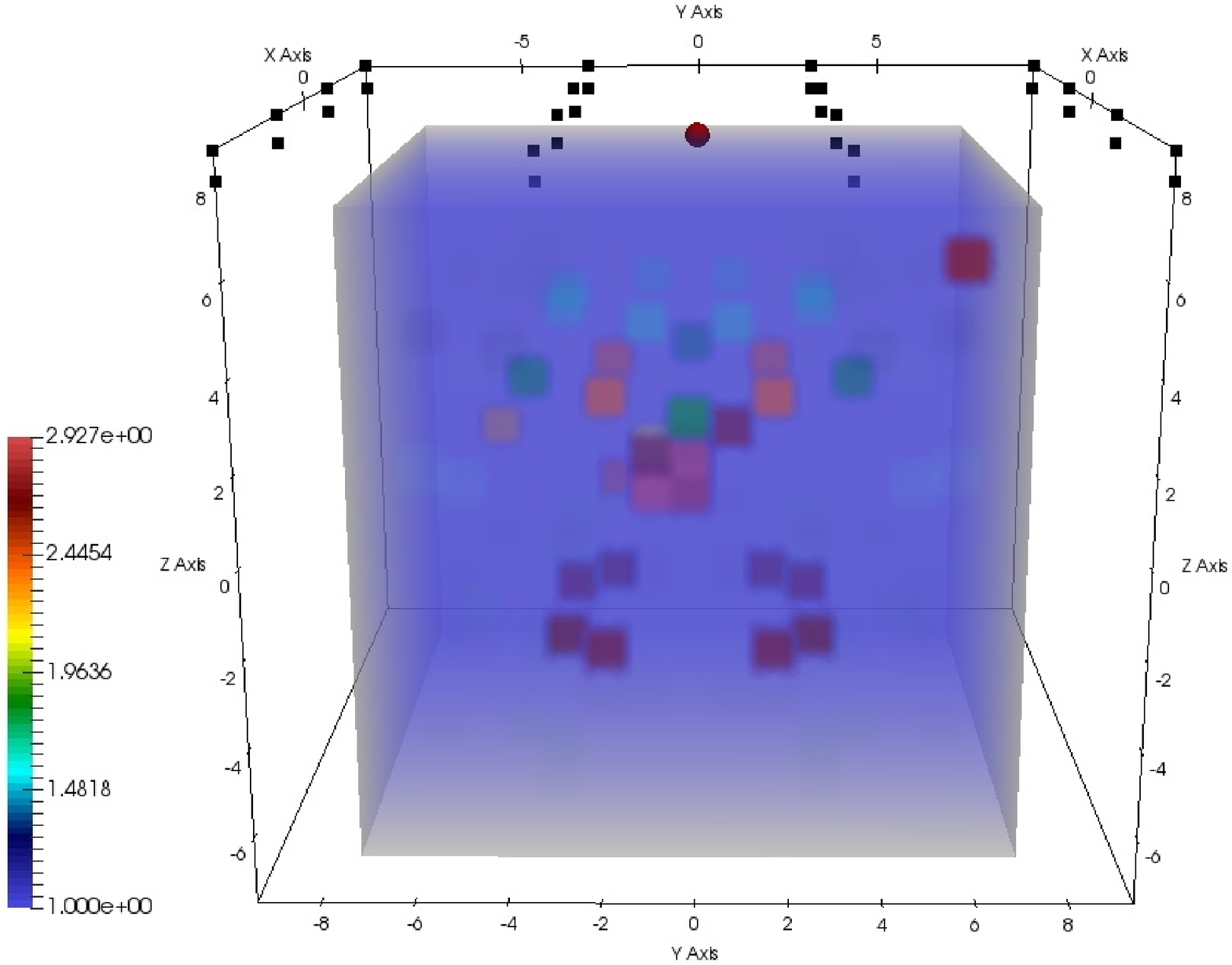}} & \resizebox*{6cm}{!}{\includegraphics{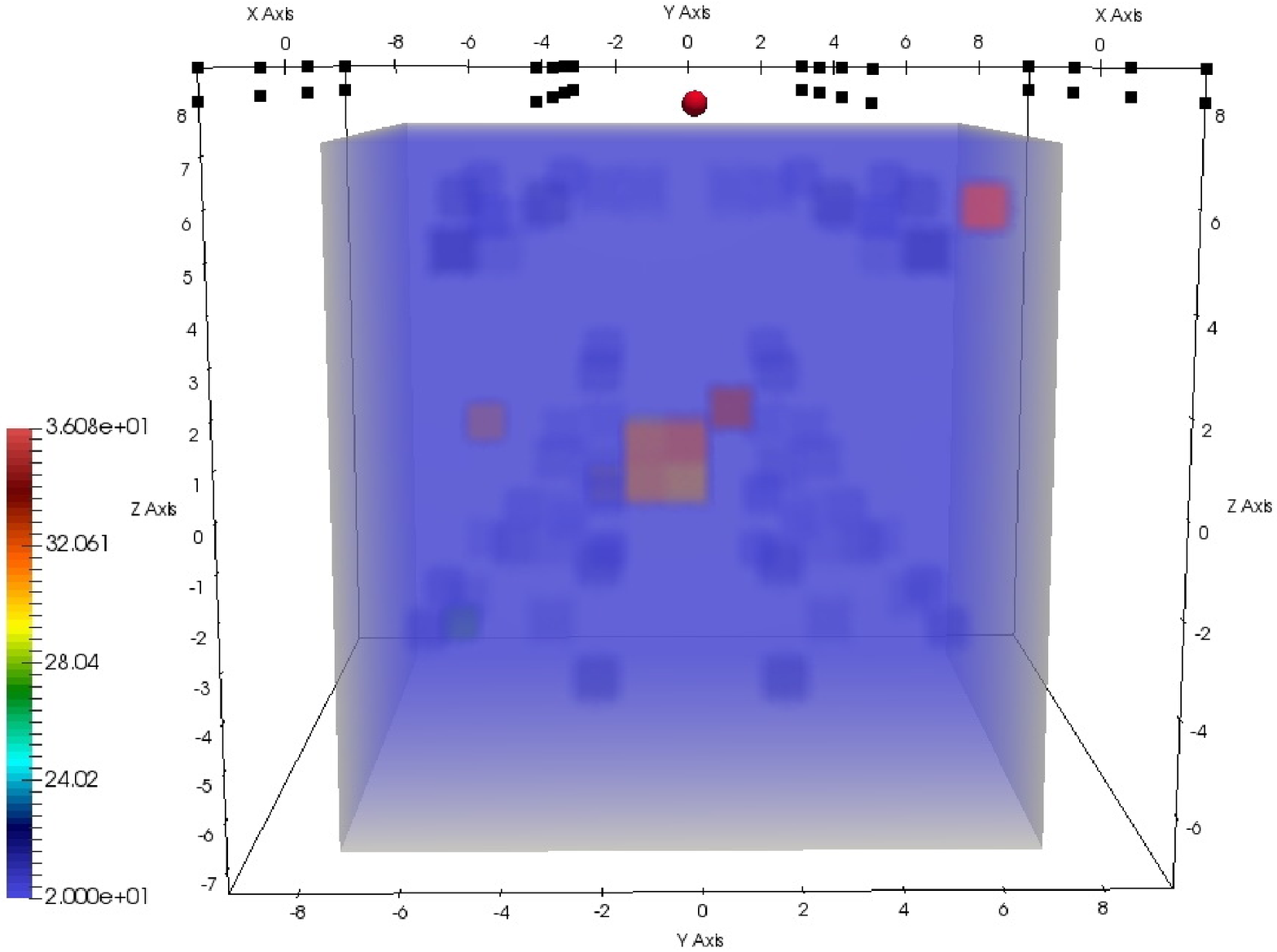}}\\
(a) & (b)
\end{tabular}
\caption{Approximate solution of SP: $d_r=0.05 m$, the receivers are in planes parallel to $xy-$plane, the input data is noisy.}
\label{fig3}
\end{center}
\end{figure}


In the fourth experiment (see Figures \ref{fig4}(a) and \ref{fig4}(b)), the distance $d_r$ equals $0.005 ,m$ the input data is noisy, and the location of the receivers is changed.
\begin{figure}
\begin{center}
\begin{tabular}{cc}
\resizebox*{6cm}{!}{\includegraphics{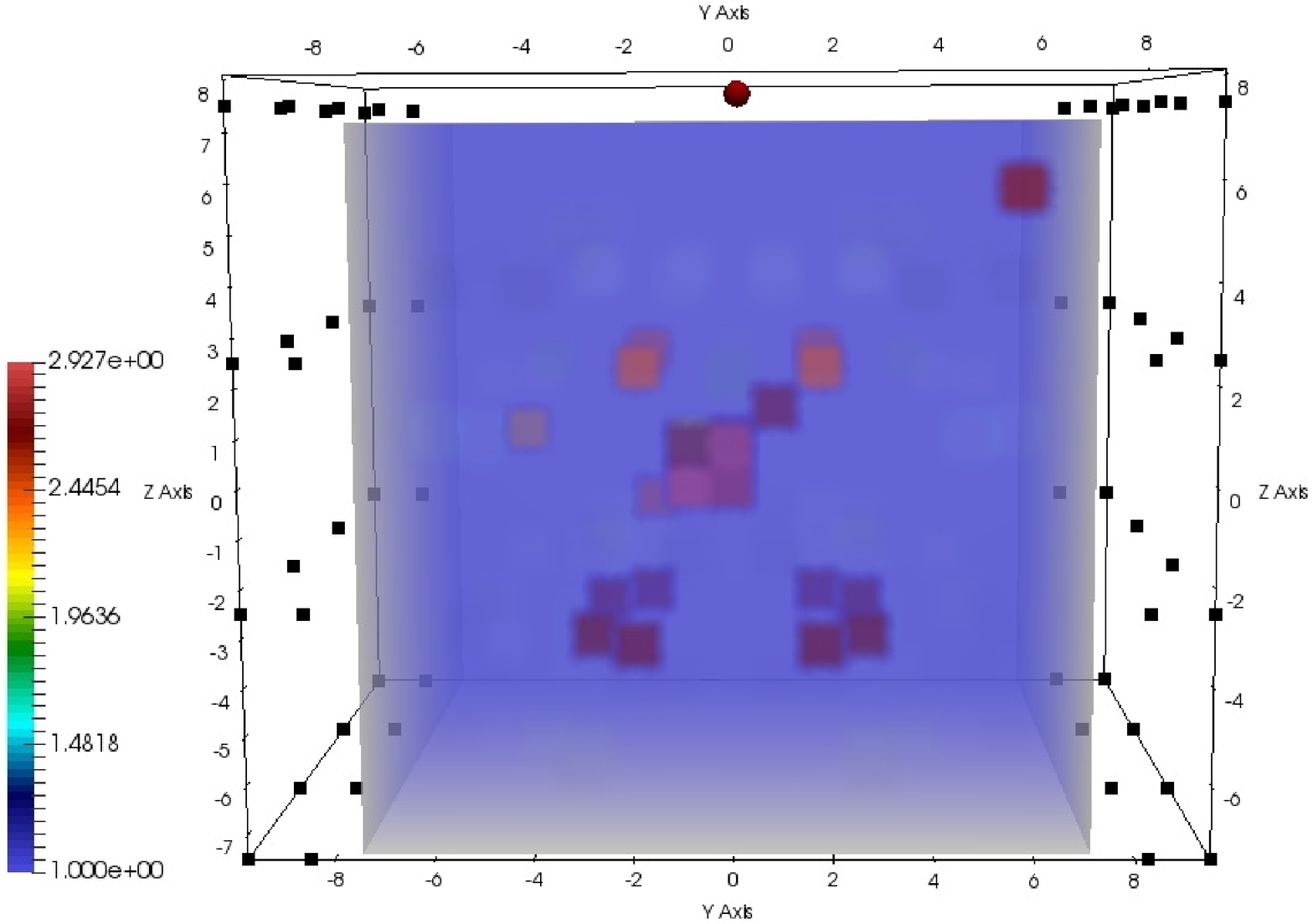}} & \resizebox*{6cm}{!}{\includegraphics{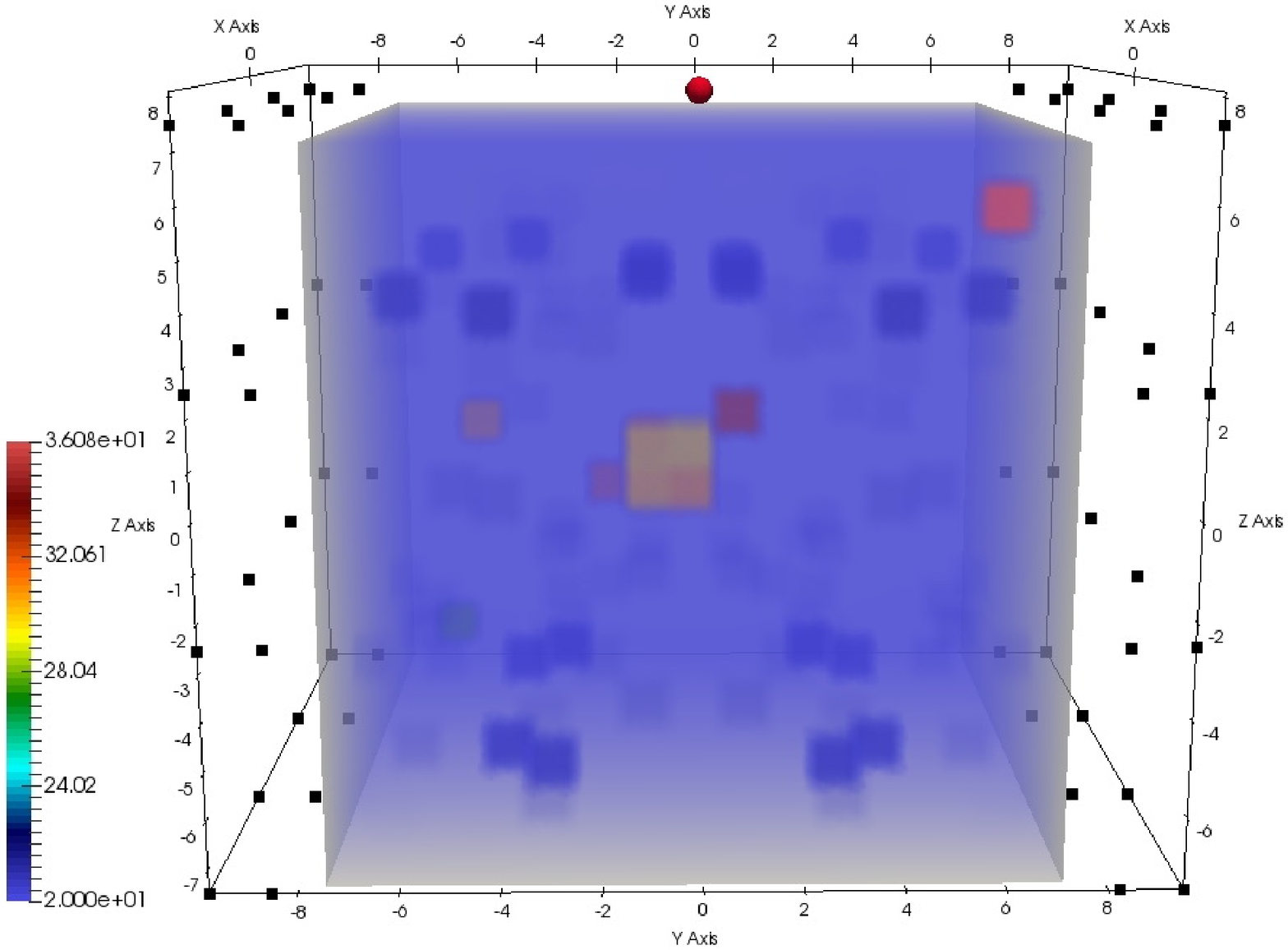}}\\
(a) & (b)
\end{tabular}
\caption{Approximate solution of SP: $d_r=0.05 m$, the receivers are in planes parallel to $xz-$plane, the input data is noisy.}
\label{fig4}
\end{center}
\end{figure}


In the final test (see Figures \ref{fig5}(a) and \ref{fig5}(b)), we changed the location of both the field source and the receivers. In this case, a more accurate solution was obtained.  Despite some "artifacts", the values of the refractive index just slightly differ from the true ones.
\begin{figure}
\begin{center}
\begin{tabular}{cc}
\resizebox*{6cm}{!}{\includegraphics{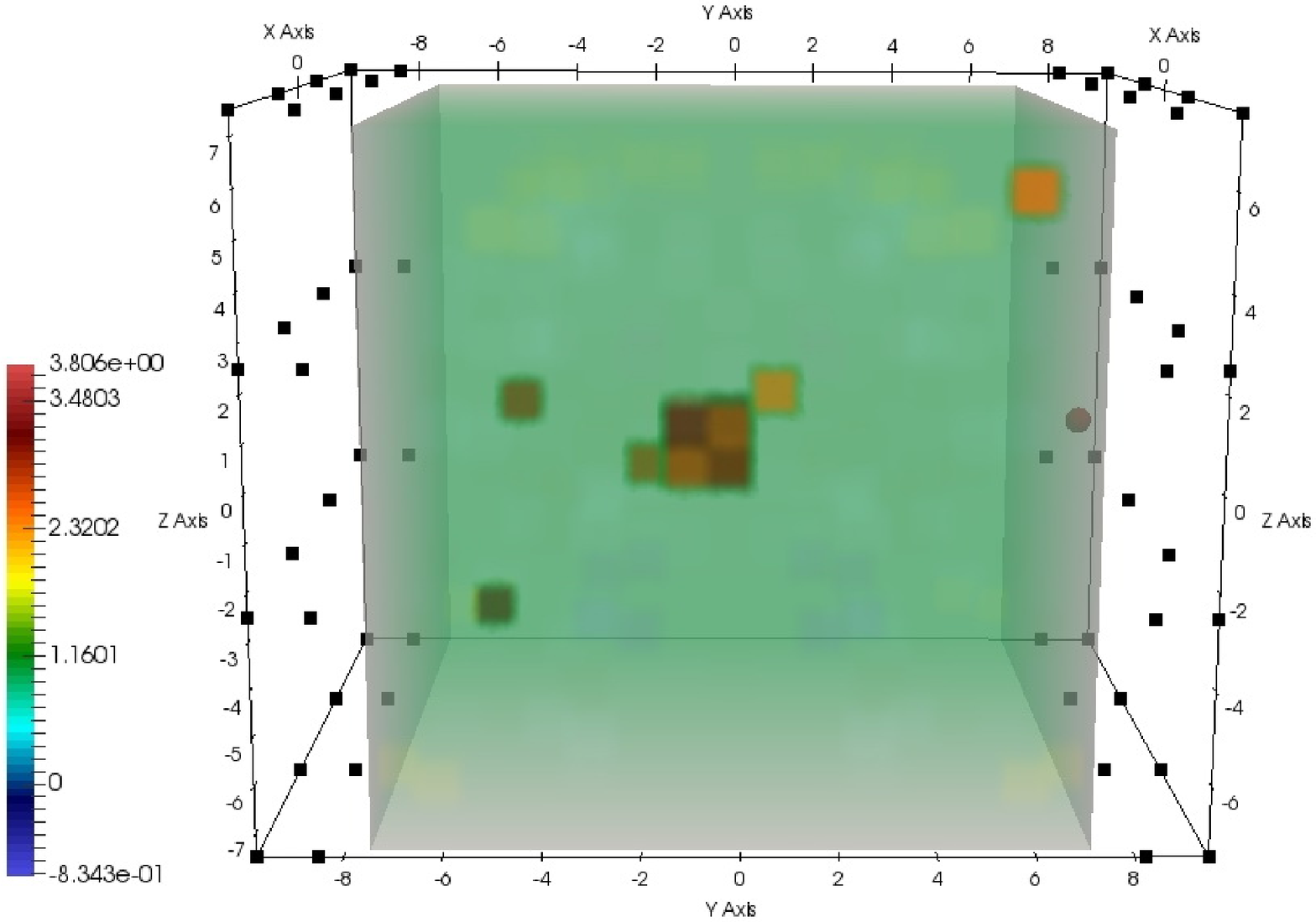}} & \resizebox*{6cm}{!}{\includegraphics{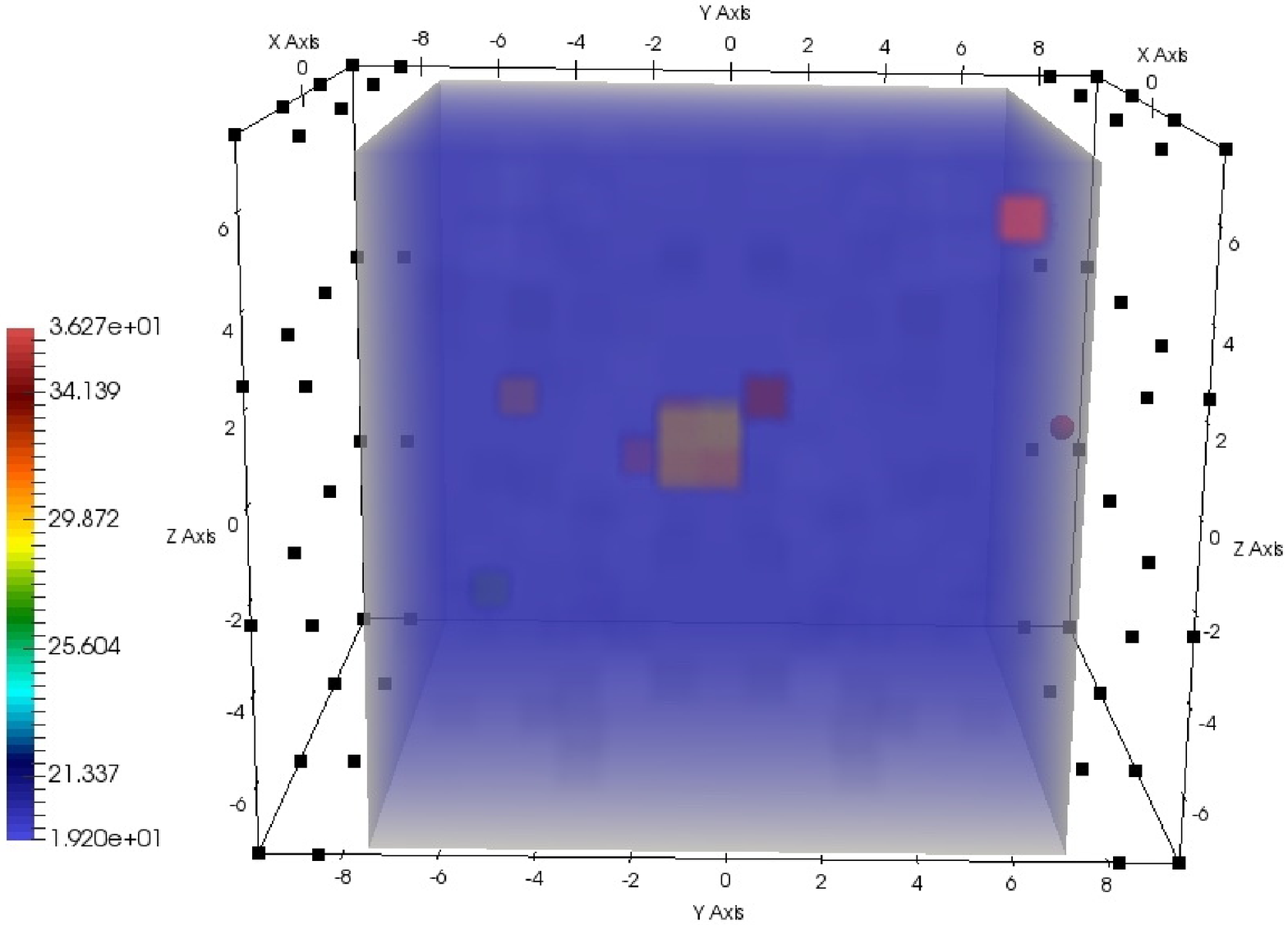}}\\
(a) & (b)
\end{tabular}
\caption{Approximate solution of SP: $d_r=0.05 m$, the receivers are in planes parallel to $xz-$plane, the input data is noisy.}
\label{fig5}
\end{center}
\end{figure}

\section*{Conclusion}


The authors developed and theoretically justified a new method for solving the problem of reconstructing the refractive index in an inhomogeneous body using near-field data. The method involves solving a linear integral equation in the region of inhomogeneity.

The computational experiments confirm the efficiency of the developed algorithm. It is supposed to carry out experimental studies using the proposed method.

\end{document}